\documentclass[12pt]{article}

\usepackage{bm,latexsym,amsfonts,amsmath,amssymb,amsthm,url,graphics,psfrag,amscd,mathrsfs}
\usepackage{enumerate}
\usepackage{graphics}
\usepackage{graphicx}
\usepackage{color}
\usepackage{a4wide}
\usepackage{authblk} 
\usepackage{colonequals}

\newcommand{\R}{\mathbb{R}}

\newcommand{\E}{\mathbb{E}}

\newcommand{\pp}{\mathbb{P}}

\newcommand{\kO}{O}

\newcommand{\kS}{\mathcal{S}}

\newcommand{\kL}{\mathcal{L}}

\newcommand{\lin}{\left[\kern-0.15em\left[}
\newcommand{\rin} {\right]\kern-0.15em\right]}
\newcommand{\linf}{[\kern-0.15em [}
\newcommand{\rinf} {]\kern-0.15em ]}
\newcommand{\ilin}{\left]\kern-0.15em\left]}
\newcommand{\irin} {\right[\kern-0.15em\right[}

\newcommand{\Keywords}[1]{\par\indent
{\small{\textbf{Key words and phrases.} \/} #1}}
\newcommand{\subjclass}[1]{\par\indent{ \textbf{AMS subject classification.} #1}}

\newtheorem{lem}{Lemma}[section]

\newtheorem{prop}[lem]{Proposition}
\newtheorem{theo}[lem]{Theorem}

\newtheorem {rem}[lem] {Remark}

\newtheorem*{ack}{Acknowledgments}

\title{Persistence probability of a random polynomial arising from evolutionary game theory}
\author[1,3]{Van Hao Can}
\author[2] {Manh Hong Duong}
\author[1]{Viet Hung Pham}
\affil[1]{Institute of Mathematics, Vietnam Academy of Science and Technology, 18 Hoang Quoc Viet Street, 10307 Hanoi, Vietnam. Emails: cvhao89@gmail.com; pgviethung@gmail.com}
\affil[2]{School of Mathematics, University of Birmingham, Birmingham B15 2TT,  UK. Email: h.duong@bham.ac.uk}
\affil[3]{Research Institute for Mathematical Sciences, Kyoto University, Kyoto 606--8502, Japan.}

\begin{document}
\maketitle
\begin{abstract}
In this paper, we obtain an asymptotic formula for the persistence probability in the positive real line of a random polynomial arising from evolutionary game theory. It corresponds to the probability that a multi-player two-strategy random evolutionary game has no internal equilibria. The key ingredient is to approximate the sequence of random polynomials indexed by their degrees by an appropriate centered stationary Gaussian process. 
\end{abstract}
\Keywords random polynomials, evolutionary game theory, persistence probability, equilibrium points, Gaussian process.

\noindent\subjclass 30C15, 26C10, 91A22.
\section{Introduction}
\subsection{Motivation}
In this paper, we study the \textit{persistence probability}, that is the probability of not changing sign, in the positive real line of the following random polynomial 
\begin{equation}
\label{eq: fn}
f_n(x)=\sum_{i=0}^{n} \binom{n}{i}a_ix^i,
\end{equation}
where the coefficients $a_i$'s are real and independent identically distributed (i.i.d.) standard normal random variables.

Our first motivation is from random polynomial theory in which the study of zeros of a random polynomial has been studied extensively since the seminal paper of Block and P\'{o}lya \cite{BP32}. We review here relevant work on the persistence probability and refer the reader to standard monographs~\cite{BS86,farahmand1998} and recent articles~\cite{TaoVu14,NNV2016,DoVu2017,butez2017} and references therein for information on other aspects of random polynomials such as the expected number of roots, central limit theorem and large deviations.  A random polynomial can be generally expressed by\begin{equation}
P_n(x)=\sum_{i=0}^{n} c_i\,\xi_i\,x^i,
\end{equation}
where $c_i$ are deterministic coefficients which may depend on both $n$ and $i$  and $\xi_i$ are random variables. The most popular random polynomials studied in the literature are:
\begin{enumerate}[(i)]
\item Kac polynomials (denoted by $P^K_n$): $c_i\colonequals 1$,
\item Weyl (or flat) polynomials ($P^W_n$): $c_i\colonequals\frac{1}{i!}$,
\item Elliptic (or binomial) polynomials ($P^E_n$): $c_i\colonequals\sqrt{\begin{pmatrix}
n\\
i
\end{pmatrix}}$.
\end{enumerate}
For Kac polynomials, it is shown in~\cite{LittlewoodOfford1939,LittlewoodOfford1948} that, when $\{\xi_i\}_{i=0}^n$ are i.i.d. and are either all uniform on $[−1, 1]$ or all Gaussian or all uniform on $\{−1, 1\}$, $\mathbb{P}(N^K_{n}=0)=O(1/\log n)$ where $N^K_n$ is the number of real zeros of the Kac polynomial $P_n^K$. This result is extended in~\cite{Dembo2002} to the case where $\xi_i$ are i.i.d. random variables with the common distribution having finite moments of all orders as
$$
\mathbb{P}(P^K_n(x) > 0, \forall x\in\R) = n^{-4b_0+o(1)},
$$
where the constant $b_0$ above is given by
$$
b_0=-\lim\limits_{t\to\infty} t^{-1}\log\mathbb{P}(X_s>0,~\forall s\in[0,t]),$$
where $X$ is a centered stationary Gaussian process with correlation $\mathbb{E}(X_0X_t)=1/\cosh(t/2)$.   In~\cite{SM08}, the authors develop a mean-field
approximation to re-derive the persistence probability of (generalized) Kac polynomials relating it to zero crossing properties of the diffusion equation with
random initial conditions. Moreover, using this method, they predict and numerically verify the following asymptotic formulas for elliptic and  Weyl models:
\begin{enumerate}[(i)]
\item For elliptic polynomials: 
\begin{equation}
\label{eq: elliptic}
\lim\limits_{n\rightarrow\infty}\frac{\log\mathbb{P}(P^E_n(x)>0,~\forall x\in\R)}{\sqrt{n}}=-2\pi b,
\end{equation}
where $b$ is a positive constant defined as
\begin{equation}
\label{eq: b}
b=-\lim_{T\rightarrow\infty}\frac{\log\mathbb{P}(\inf_{0\leq t\leq T} Y(t))}{T},
\end{equation}
where $Y(t)$ is a centered stationary Gaussian process with correlation $\mathbb{E}(Y_0Y_t)=e^{-t^2/2}$.
\item For Weyl polynomials
\begin{equation}
\label{eq: Weyl}
\lim\limits_{n\rightarrow\infty}\frac{\log\mathbb{P}(P^W_n(x)>0,~\forall x\in\R)}{\sqrt{n}}=-2 b,
\end{equation}
with the same constant $b$ as in \eqref{eq: b}.
\end{enumerate}
The statement \eqref{eq: elliptic} for elliptic polynomials is proven in~\cite{DM}. In addition, this work also shows \eqref{eq: Weyl} for $I_n = [0,\sqrt{n}-\alpha_n]$ with $n^{−1/2}\alpha_n\rightarrow 0$ obtaining the persistence exponent $-b$. More recently, by extending the method of~\cite{DM}, the authors of~\cite{CanPham2017TMP} prove \eqref{eq: Weyl} for Weyl polynomials.  Inspired by this development, in this paper we study the asymptotic behaviour of the persistence probability in the positive real line of the random polynomial $f_n$ in \eqref{eq: fn}. This is a new class of random polynomials and as will be discussed in the next paragraph, the persistence probability of $f_{n-1}$ corresponds to the probability that an $n$-player two-strategy random evolutionary game has no internal equilibria. 

Our second motivation comes from evolutionary game theory~ \cite{SP73,hofbauer:1998mm}. As will be shown in Appendix \ref{sec: replicator}, under certain assumptions, the polynomial $f_n$ originates from the study of equilibrium points in random evolutionary game theory: finding an internal equilibrium point in a symmetric $n$-player two-strategy random game is equivalent to finding a positive zero of $f_{n-1}$. In particular, the persistence probability of $f_{n-1}$ in the positive real line corresponds to the probability that the random game has no internal equilibria.   Random evolutionary games have been used widely and successfully in the mathematical modelling of social and biological systems where limited
information is available or where the environment changes so rapidly and frequently that one
cannot predict the payoffs of their inhabitants. Such scenarios arise in fields as biology, ecology, population genetics, economics and social sciences~\cite{may2001stability,fudenberg1992evolutionary,HTG12,
gross2009generalized}.  In these situations, due to randomness, characterizing the statistical properties of equilibrium points becomes essential and has attracted considerable interest in recent years. In~\cite{gokhale:2010pn,HTG12,gokhale2014evolutionary}, the authors provide analytical and simulation results for random games with a small number of players ($n\leq 4$) focusing on the probability of attaining the maximal number of equilibrium points. In~\cite{DH15,DuongHanJMB2016,DuongTranHan2017a}, the authors derive a closed formula for the expected number of internal equilibria, characterize its asymptotic behaviour and study the effect of correlations. Related work on the expected number of equilibrium points of random large complex systems arising from physics and ecology are presented in~\cite{Fyo2004,FyoNadal2012,FyoKho2016}, see also references therein. More recently, \cite{DuongTranHan2017b} offers, among other things, an analytical formula for the probability that a multi-player two-strategy game has a certain number of internal equilibria. Although the analytical formula is theoretically interesting, it involves complicated multiple integrals and is computationally intractable when the number of player becomes large. The present paper provides an asymptotic formula, as the number of players tends to infinity, for the probability that the game has no internal equilibria. Biologically this probability corresponds to the two extreme cases when the whole population consists of only one specie/strategy while the other extincts.
\subsection{The main result of the paper} 
The main result of the present paper is the following theorem.
\begin{theo}\label{mth}
Let $f_n$ be defined in \eqref{eq: fn} where the coefficients $a_i$ are i.i.d. standard normal random variables. Then we have
\begin{equation}
\label{eq: main thm}
\lim_{n \rightarrow \infty} \frac{\log \pp \left( f_n(x) >0 , \, \forall x\in (0,\infty)\right)}{\pi \sqrt{n}}=-b,
\end{equation}
where $b$ is the persistence exponent defined by
\begin{equation}\label{exp}
b= - \lim_{T \rightarrow \infty} \frac{\log \pp (\inf_{0 \leq t \leq T} Z(t) >0)}{T} ,
\end{equation}
where $Z(t)$ is a centered stationary Gaussian process with correlation $\E (Z_0Z_t)=e^{-t^2/4}$.
\end{theo}
The idea of the proof is as follows. We first show that the contributions to the persistence exponent of intervals $(0,n^{-1/6})$ and $(n^{1/6},\infty)$ are negligible. We then apply the method of \cite{DM} to prove that the main contribution from the interval $(n^{-1/6},n^{1/6})$ can be calculated approximately from that of a centered stationary Gaussian process with autocorrelation function $R(t)=e^{-t^2/4}$.  
\subsection{Organization of the paper}
The rest of paper is organized as follows. Section \ref{sec: prel} contains technical lemmas. The proof of the main theorem is presented in Section \ref{sec: proof}. We provide further discussion on future work in Section \ref{sec: final}. In Appendix \ref{sec: replicator}, we show the derivation of the random polynomial $f_n$ from the replicator dynamics for a symmetric multi-player two-strategy random evolutionary game.
\section{Preliminaries}
\label{sec: prel}
In this section, we prove some technical results that will be used in the proof of the main theorem presented in Section \ref{sec: proof}. 

Since  $\{a_i\}$ are i.i.d. random variables of standard normal distribution, the  random polynomial $f_n(x)$ is a  Gaussian process with autocorrelation function
\begin{equation}
A_n(x,y)= \frac{M_n(\sqrt{xy})}{\sqrt{M_n(x)}\sqrt{M_n(y)}},
\end{equation}
where 
\begin{equation}
\label{eq: Mn}
M_n(x)=\sum \limits_{i=0}^{n} \binom{n}{i}^2x^{2i}.
\end{equation}
We prove here a key lemma on the behavior of $M_n(x)$ as $n\rightarrow \infty$. 
\begin{lem} \label{lom} For $x\in (0,1]$, we define 
$$i_x=\left[ \frac{nx}{x+1}\right].$$
\begin{itemize}
\item[(i)]  For $\tfrac{\log n}{6n} \leq x \leq 1 $, we have 
\begin{eqnarray*}
\binom{n}{i_x}^2x^{2i_x} \leq M_n(x) \leq 3 i_x^{3/4} \binom{n}{i_x}^2x^{2i_x}.
\end{eqnarray*}
\item[(ii)]  For $n^{-1/6} \leq x \leq 1 $, we have 
\begin{eqnarray*}
M_n(x)&=& (1+\kO(n^{-1/24}))  \binom{n}{i_x}^2x^{2i_x}\frac{\sqrt{\pi n x}}{(x+1)}\\
&=&(1+\kO(n^{-1/24}))\frac{(x+1)^{2n+1}}{2\sqrt{\pi n x}}.
\end{eqnarray*}
\end{itemize}

\end{lem}
\begin{proof} The proof of this lemma is fairly lengthy and technical. The main idea is to use  Stirling formula to approximate the summand in $M_n(x)$ by another function $J_x\Big(\frac{i}{n}\Big)$, see Eq. \eqref{doj} below. The tasks are then to understand the behaviour of $J_x\Big(\frac{i}{n}\Big)$ which is subtly dependent on the relationship between $x$ and $n$. It turns out that the behaviour of $M_n(x)$ depends on whether $x\in[\tfrac{\log n}{6n},1]$ or $x\in[n^{-1/6},1]$ as in Part (i)/(ii). This will be carried out using an intermediate parameter $i_x$ defined at the beginning of the lemma.

Let us start with Stirling formula that 
\begin{eqnarray*}
i! = \sqrt{2\pi i}(1+\kO(i^{-1})) \left(\frac{i}{e}\right)^i.
\end{eqnarray*}
Therefore,
\begin{eqnarray*}
\binom{n}{i}&=&\sqrt{\frac{n}{2\pi i(n-i)}} \left(1+\kO(\max(i^{-1},(n-i)^{-1}))\right) \left(\frac{n}{i}\right)^i \left(\frac{n}{n-i}\right)^{n-i}\\
&=& \sqrt{\frac{n}{2\pi i(n-i)}} \left(1+\kO(\max(i^{-1},(n-i)^{-1}))\right) \exp \left( n I\left(\frac{i}{n}\right) \right), 
\end{eqnarray*}
where $I(0)=I(1)=0$ and for $t\in(0,1)$,
$$I(t)=(t-1) \log (1-t)-t \log t.$$
Hence
\begin{eqnarray} \label{doj}
\binom{n}{i}^2 x^{2i} = \frac{n}{2\pi i(n-i)} \left(1+\kO(\max(i^{-1},(n-i)^{-1}))\right) \exp \left( 2n  J_x\left(\frac{i}{n}\right) \right),
\end{eqnarray}
where 
$$J_x(t)=I(t)+ t \log x.$$
We notice that 
\begin{equation} \label{coc}
J_x\left(\frac{x}{x+1}\right)= \log (x+1), \quad J'_x\left(\frac{x}{x+1}\right)= 0, \quad  J''_x(t)= \frac{-1}{t(1-t)} \,\,\forall \,\, t\in(0,1).   
\end{equation}
Therefore, using Taylor expansion, we get 
\begin{eqnarray*}
J_x\left(\frac{i_x}{n}\right) &=& J_x\left(\frac{x}{x+1}\right)  + J_x'\left(\frac{x}{x+1}\right) \left( \frac{i_x}{n} -\frac{x}{x+1} \right) + \frac{J_x''(\eta_x)}{2} \left( \frac{i_x}{n} -\frac{x}{x+1} \right)^2\\
&=& \log (x+1) + \frac{J_x''(\eta_x)}{2} \left( \frac{i_x}{n} -\frac{x}{x+1} \right)^2,
\end{eqnarray*}
with some $\eta_x \in \left( \tfrac{i_x}{n}, \tfrac{x}{x+1} \right)$.  Using the assumption that $x\in (0, 1]$, we have 
\begin{eqnarray*}
\big | \frac{i_x}{n} - \frac{x}{x+1} \big | \leq \frac{1}{n}, \quad |J''(\eta_x)| = \frac{1}{\eta_x(1-\eta_x)}  \leq \frac{2n}{i_x}.
\end{eqnarray*}
Hence
\begin{eqnarray*}
2n\left(J_x\left(\frac{i_x}{n}\right) - \log(x+1) \right) \leq \frac{2}{i_x}. 
\end{eqnarray*}
Therefore, using the fact that $\max(i_x^{-1},(n-i_x)^{-1})=i_x^{-1}$, we get
\begin{eqnarray}
\binom{n}{i_x}^2 x^{2i_x} &=& \frac{n}{2\pi i_x(n-i_x)} \left(1+\kO(i_x^{-1})\right) \exp \left( 2n  J_x\left(\frac{i_x}{n}\right) \right) \label{eix}\\
&=& \frac{ (x+1)^2}{2\pi n x}\left(1+\kO(i_x^{-1})\right) (x+1)^{2n} \notag\\
&=&  \left(1+\kO(i_x^{-1}) \right)\frac{(x+1)^{2n+2}}{2 \pi nx}. \notag
\end{eqnarray}
We now estimate $M_n(x)$. Observe that 
\begin{eqnarray}
M_n(x) = \sum \limits_{i=0}^n \binom{n}{i}^2 x^{2i} = \binom{n}{i_x}^2 x^{2i_x}  \sum \limits_{i=0}^n \frac{\binom{n}{i}^2 x^{2i}}{\binom{n}{i_x}^2 x^{2i_x}}. 
\end{eqnarray}
We notice that for any $i=0,\ldots,n$
\begin{eqnarray*}
\binom{n}{i} \leq e^{nI(i/n)}.
\end{eqnarray*}
Thus  for any $i=0,\ldots,n$
\begin{eqnarray} \label{ljx}
\binom{n}{i}^2 x^{2i} \leq e^{2nJ_x(i/n)}.
\end{eqnarray}

{\it Observation (O1).} By \eqref{coc}, the function $J_x(t)$ is concave in $(0,1)$  and attains the maximum at $t=\tfrac{x}{x+1}$. Thus for any closed interval $A\subset (0,1)$, 
\begin{equation*}
\max \limits_{t \in A} J_x(t)= \max \limits_{t \in A_x} J_x(t),
\end{equation*}
with 
$$A_x=\{t: |t-\tfrac{x}{x+1}|=\min \limits_{s \in A} |s-\tfrac{x}{x+1}|\}.$$
{\bf Case 1. $i\leq i_x -i_x^{3/4}$, or $i_x+ i_x^{3/4} \leq i \leq  60 i_x$}. Then  by \eqref{eix},\eqref{ljx} and observation (O1), we have
\begin{eqnarray*}
\frac{\binom{n}{i}^2 x^{2i}}{\binom{n}{i_x}^2 x^{2i_x}} &\leq & 4\pi i_x  \exp \left( 2n \left[ J_x\left(\frac{i}{n}\right) -J_x\left(\frac{i_x}{n}\right) \right] \right) \notag \\
&\leq & 4\pi i_x  \exp \left( 2n \left[ J_x\left(\frac{i_x \pm i_x^{3/4}}{n}\right) -J_x\left(\frac{i_x}{n}\right) \right] \right).  \label{iix}
\end{eqnarray*}
 By Taylor expansion, 
\begin{eqnarray*}
J_x\left(\frac{i_x \pm i_x^{3/4}}{n}\right) -J_x\left(\frac{i_x}{n}\right)=\pm \frac{i_x^{3/4}}{n} J'_x\left(\frac{i_x}{n}\right) + J''_x(\nu_x) \frac{i_x^{3/2}}{2n^2}, 
\end{eqnarray*}
for some $\nu_x \in (\tfrac{i_x-i_x^{3/4}}{n},\tfrac{i_x+i_x^{3/4}}{n})$. Notice that
\begin{eqnarray}
\Big |J'_x\left(\frac{i_x}{n}\right) \Big |= \Big |J'_x\left(\frac{i_x}{n}\right)-J'_x\left(\frac{x}{x+1}\right) \Big | &\leq& \sup \limits_{y \in (\tfrac{i_x}{n}, \tfrac{x}{x+1})} |J''(y)| \Big |\frac{i_x}{n}-\frac{x}{x+1} \Big | \notag \\
&\leq& \frac{4}{nx}, \label{jpx}
\end{eqnarray}
by using \eqref{coc} and $|\tfrac{i_x}{n}-\tfrac{x}{x+1}|\leq \tfrac{1}{n}$. On the other hand,
\begin{eqnarray*}
J_x''(\nu_x)=\frac{-1}{\nu_x(1-\nu_x)} \leq \frac{-1}{\nu_x} \leq \frac{-n}{i_x+i_x^{3/4}} \leq \frac{-1}{2x}.
\end{eqnarray*}
Combining above estimates, we get
\begin{eqnarray*}
\frac{\binom{n}{i}^2 x^{2i}}{\binom{n}{i_x}^2 x^{2i_x}} &\leq & 4\pi i_x  \exp \left( \frac{8 i_x^{3/4}}{nx} - \frac{i_x^{3/2}}{2nx}\right) \leq 4\pi i_x  \exp \left(  - \frac{i_x^{3/2}}{4nx}\right) \leq 4\pi i_x e^{-\tfrac{\sqrt{2nx}}{16}}.
\end{eqnarray*}
Therefore, when $nx$ is large enough,
\begin{eqnarray}
M_{1,n}(x) &=& \sum \limits_{i\leq i_x -i_x^{3/4}} \frac{\binom{n}{i}^2 x^{2i}}{\binom{n}{i_x}^2 x^{2i_x}} + \sum \limits_{ i_x +i_x^{3/4} \leq i \leq 60 i_x} \frac{\binom{n}{i}^2 x^{2i}}{\binom{n}{i_x}^2 x^{2i_x}} \notag \\
&\leq& 61i_x \times 4\pi i_x e^{-\tfrac{\sqrt{2nx}}{16}} \leq 244 \pi (nx)^2 e^{-\tfrac{\sqrt{2nx}}{16}} \leq \frac{1}{nx}. \label{m1n}
\end{eqnarray}

{\bf Case 2. $i>60i_x$}. Using the same arguments as in Case 1, we can show that 
\begin{eqnarray*}
\frac{\binom{n}{i}^2 x^{2i}}{\binom{n}{i_x}^2 x^{2i_x}} &\leq &  4\pi i_x  \exp \left( 2n \left[ J_x\left(\frac{60i_x }{n}\right) -J_x\left(\frac{i_x}{n}\right) \right] \right)\\
&\leq& 4\pi i_x  \exp \left( 2n \left[ \frac{4}{nx} \frac{59i_x}{n} - \frac{1}{120 i_x} \frac{(59i_x)^2}{2n^2} \right]  \right) \\
&\leq & 4\pi i_x  \exp \left(472 - \frac{59^2 nx}{240} \right)\\
&\leq & 4\pi i_x  \exp \left(472 - \frac{59^2 \log n}{1440} \right) \leq 4 \pi i_x e^{-12 \log n/5} = 4 \pi i_x n^{-12/5}.
\end{eqnarray*}
Notice that for the last line, we assume that $n$ is large enough and $nx \geq \log n/6$. 
Therefore,
\begin{eqnarray}
M_{2,n}(x) &=&  \sum \limits_{  i > 60 i_x} \frac{\binom{n}{i}^2 x^{2i}}{\binom{n}{i_x}^2 x^{2i_x}} \leq n \times 4\pi i_x n^{-12/5} \leq n^{-1/5}. \label{m2n}
\end{eqnarray}

{\bf Case 3. $|i-i_x| \leq i_x^{3/4}$}.  By \eqref{doj}, noting that in this case $i$ and $i_x$ are of the same order of magnitude, we have  
\begin{eqnarray}
&& \frac{\binom{n}{i}^2 x^{2i}}{\binom{n}{i_x}^2 x^{2i_x}} =  (1+\kO(i_x^{-1})) \frac{i_x(n-i_x)}{i(n-i)}  \exp \left( 2n \left[ J_x\left(\frac{i}{n}\right) -J_x\left(\frac{i_x}{n}\right) \right] \right). \label{isix}
\end{eqnarray}
Since $J_x(t)$ is a concave function, 
\begin{eqnarray}
J_x\left(\frac{i}{n}\right) -J_x\left(\frac{i_x}{n}\right) \leq J_x'\left(\frac{i_x}{n}\right) \frac{(i-i_x)}{n}. \label{cbn}
\end{eqnarray}
On the other hand, by \eqref{jpx} 
\begin{eqnarray} \label{jxp}
|J_x'\left(\tfrac{i_x}{n}\right) (i-i_x)| \leq \frac{4|i-i_x|}{nx} \leq \frac{4 i^{3/4}_x}{nx} \leq \frac{4}{(i_x)^{1/4}}.
\end{eqnarray}
We are now in the position to prove (i). Indeed, from \eqref{isix}, \eqref{cbn} and \eqref{jxp}, we have
\begin{eqnarray*}
&& \frac{\binom{n}{i}^2 x^{2i}}{\binom{n}{i_x}^2 x^{2i_x}} \leq  (1+i_x^{-1/2}) \frac{i_x(n-i_x)}{i(n-i)}  \exp \left( 8i_x^{-1/4} \right) \leq (1+i_x^{-1/8}).
\end{eqnarray*} 
Hence,
\begin{eqnarray*}
&& 1\leq M_{3,n}(x) =\sum \limits_{|i-i_x| \leq i_x^{3/4}} \frac{\binom{n}{i}^2 x^{2i}}{\binom{n}{i_x}^2 x^{2i_x}} \leq  2i_x^{3/4}  (1+i_x^{-1/8}).
\end{eqnarray*} 
Combining this estimate with \eqref{m1n} and \eqref{m2n}, we obtain (i).

We now prove (ii). Assume that $x\in(n^{-1/6},1)$. Then using \eqref{isix} and Taylor expansion,
\begin{eqnarray*}
&& \frac{\binom{n}{i}^2 x^{2i}}{\binom{n}{i_x}^2 x^{2i_x}} =  (1+\kO(i_x^{-1})) \frac{i_x(n-i_x)}{i(n-i)}  \exp \left( 2n \left[ J_x\left(\frac{i}{n}\right) -J_x\left(\frac{i_x}{n}\right) \right] \right)\\
&=& (1+\kO(i_x^{-1/4}))  \exp \left( 2n \left[ J_x'\left(\frac{i_x}{n}\right) \frac{(i-i_x)}{n}+  J_x''\left(\frac{i_x}{n}\right) \frac{(i-i_x)^2}{2n^2} + J_x'''\left(\nu_{i,x}\right) \frac{(i-i_x)^3}{6n^3}  \right]  \right),
\end{eqnarray*}
for some $\nu_{i,x} \in (\tfrac{i_x}{n}, \tfrac{i}{n})$.  We notice that 
\begin{eqnarray*}
J_x'''(y) = \kO(y^{-2})\quad\text{for all}~~y\in \mathbb{R}.
\end{eqnarray*}
Therefore, by Taylor expansion, 
\begin{equation*}
J_x''\left(\frac{i_x}{n}\right)=J_x''\left(\frac{x}{x+1}\right) + \kO\left(\frac{1}{x^2}\right) \left(\frac{i_x}{n}-\frac{x}{x+1}\right) =J_x''\left(\frac{x}{x+1}\right) + \kO\left(\frac{1}{nx^2}\right),
\end{equation*}
and 
\begin{equation*}
J_x'''\left(\nu_{i,x}\right)=\kO\left(\frac{1}{x^2}\right).
\end{equation*}
Combining these estimates with \eqref{jxp}, we get 
\begin{eqnarray*}
  \tfrac{\binom{n}{i}^2 x^{2i}}{\binom{n}{i_x}^2 x^{2i_x}}  &=& (1+\kO(i_x^{-1/4}))  \exp \left( \kO(i_x^{-1/4}) + \left[   J_x''\left(\tfrac{x}{x+1}\right)  +  \kO\left(\tfrac{1}{nx^2}\right) + \kO\left( \tfrac{1}{x^2}\right) \left(\tfrac{i-i_x}{n}\right)\right]\tfrac{(i-i_x)^2}{n}   \right), \\
  &=& (1+\kO(n^{-5/24}))  \exp \left(  \left[   J_x''\left(\tfrac{x}{x+1}\right)  +  \kO\left(n^{-1/24}\right)\right]\tfrac{(i-i_x)^2}{n}   \right), 
\end{eqnarray*}
since $x\in(n^{-1/6},1)$ and $|i-i_x|\leq i_x^{3/4} =\kO((nx)^{3/4})$. Therefore, by using $J_x''(\tfrac{x}{x+1})=\tfrac{-(x+1)^2}{x}$ and integral approximations, we can prove that 
\begin{eqnarray*}
M_{3,n}(x)&=& \sum \limits_{|i-i_x|\leq i_x^{3/4}}   \tfrac{\binom{n}{i}^2 x^{2i}}{\binom{n}{i_x}^2 x^{2i_x}} \\
&=& (1+\kO(n^{-5/24})) \sum \limits_{|j|\leq i_x^{3/4}}  \exp \left(  \left[ \tfrac{-(x+1)^2}{x} +\kO(n^{-1/24}) \right]  \tfrac{j^2}{n}   \right)\\
&=& (1+\kO(n^{-5/24})) \sum \limits_{|j|\leq \left(\tfrac{nx}{x+1}\right)^{3/4}}  \exp \left(   -  \left( j \tfrac{\sqrt{\tfrac{(x+1)^2}{x}+\kO(n^{-1/24})}}{\sqrt{n}} \right)^2   \right) \\
&=& (1+\kO(n^{-5/24})) \tfrac{\sqrt{n}}{\sqrt{\tfrac{(x+1)^2}{x}+\kO(n^{-1/24})}}  \left( \int _{-\infty}^{\infty} e^{-t^2}dt + \kO(e^{-(nx)^{1/4}}) \right) \\
&=& (1+\kO(n^{-1/24})) \frac{\sqrt{\pi nx}}{x+1}.
\end{eqnarray*}
In conclusion, we have
\begin{eqnarray*}
\sum \limits_{i=0}^n \binom{n}{i}^2 x^{2i} &=&\binom{n}{i_x}^2 x^{2i_x} (M_{1,n}(x)+M_{2,n}(x)+M_{3,n}(x)) \\
 &=&(1+\kO(n^{-1/24})) \binom{n}{i_x}^2 x^{2i_x} \frac{\sqrt{\pi nx}}{x+1}\\
  &=&(1+\kO(n^{-1/24})) \frac{(x+1)^{2n+1}}{2\sqrt{\pi nx}}.
\end{eqnarray*}
Part (ii) follows.
\end{proof}
\begin{rem}[Asymptotic behaviour of $M_n$ via the Legendre polynomial] The asymptotic formula of $M_n$ can also be calculated using Legendre polynomials as follows. Legendre polynomials, denoted by $L_n(x)$, are solutions to Legendre's differential equation
\begin{equation*}
\frac{d}{dx}\left[(1-x^2)\frac{d}{dx}L_n(x)\right]+n(n+1)L_n(x)=0,
\end{equation*}
with initial data $L_0(x)=1, \ \ L_1(x)=x$.  They have been used widely in physics and engineering and have many interesting properties, see \cite{BO99} for more information. For instance, $L_n$ has the following explicit representation 
\begin{equation*}
L_n(x)=\frac{1}{2^n}\sum_{i=0}^{n}\begin{pmatrix}
n\\
i
\end{pmatrix}^2(x-1)^{n-i}(x+1)^i.
\end{equation*}
According to~\cite[Lemma 3]{DuongHanJMB2016}, the polynomial $M_n$ defined in \eqref{eq: Mn} and the Legendre polynomial $L_n$ satisfy the following relation
\begin{equation}
\label{eq: Mn and Legendre}
M_n(x)=(1-x^2)^n L_n\Big(\frac{1+x^2}{1-x^2}\Big).
\end{equation}
According to \cite[Example 2, page 229]{BO99} (see also~\cite{Wang12}), the Legendre polynomial $L_n$ satisfies the following asymptotic behaviour as $n\to\infty$ for any $x>1$,
\begin{equation}
\label{eq: asymptotic of Pn}
L_n(x)\sim \frac{1}{\sqrt{2\pi n}}\frac{\Big(x+\sqrt{x^2-1}\Big)^{n+\frac{1}{2}}}{(x^2-1)^{\frac{1}{4}}},\qquad\text{for}~~ x>1.
\end{equation}
From \eqref{eq: Mn and Legendre} and \eqref{eq: asymptotic of Pn}, we obtain the following asymptotic behaviour for $M_n(x)$ as $n\to\infty$ for any $0<x<1$
\begin{equation*}
M_n(x)\sim \frac{(x+1)^{2n+1}}{2 \sqrt{\pi n x}}.
\end{equation*}
This is the result obtained in  Part (ii) of Lemma \ref{lom}. However, that part provided a stronger statement offering a quantitative estimate.
\end{rem}

 By transforming $x=\tan^2(t/2\sqrt{n})$ and $y =\tan^2(s/2\sqrt{n})$, we will show that for  $x, y \in [n^{-1/6}, n^{1/6}]$, the autocorrelation $A_n(x,y)$ is close to  $e^{-(t-s)^2/4}$.   It means that the sequence of  random polynomials  $(f_n(x))$ converges weakly to the centered stationary Gaussian process $Z(t)$  with covariance function $R(t)=e^{-t^2/4}$. Then by heuristic arguments,   the persistence probability of $f_n$ should tend to the corresponding one of $Z(t)$.  To ensure the continuity of persistence exponents, we need some restrictive conditions on the  autocorrelation function. The following result which is a combination of Theorem 1.6 in \cite{DM} and Lemma 3.1 in \cite{DM2016} gives us such conditions.

\begin{lem} \label{lem1}  Let $\kS_{+}$ be the class of all non-negative autocorrelation functions. Then the following statements hold.
\begin{itemize}
\item[(a)] For a centered stationary Gaussian process $\{Z_t\}_{t\geq 0}$ of autocorrelation function $
A(s, t) = A(0, t-s) \in \kS_+$, the  nonnegative limit 
$$b(A)= - \lim_{T \rightarrow \infty} \frac{\log \pp (\inf_{0 \leq t \leq T} Z(t) >0)}{T} ,$$
exists.
\item[(b)]  Let $\{Z_t^{(k)}\}_{t\geq 0},\, 1\leq k \leq \infty$ be a sequence of centered Gaussian processes of unit variance and nonnegative autocorrelation functions $A_k(s,t)$, such that $A_{\infty}(s,t) \in \kS_+$. Assume that the following conditions hold
\begin{itemize}
\item[(b1)] We have $A_k(s,s+\tau)\rightarrow A_{\infty}(0,\tau)$ uniformly in $s\geq 0$, when $k \rightarrow \infty$.
\item[(b2)] For some $\alpha >1$,
\begin{equation*}\label{con1}
\underset{k,\tau \rightarrow \infty}{\lim\sup} \, \underset{s\geq 0}{\sup} \left\{\tau^{\alpha}\log A_k(s,s+\tau) \right\} <-1.
\end{equation*}
\item[(b3)] There exists $\eta>1$ such that
\begin{equation*}\label{con3}
\underset{u\downarrow 0}{\lim\sup}|\log u|^{\eta} \underset{1\leq k \leq \infty}{\sup}p_k^2(u)<\infty,
\end{equation*}
where $p_k^2(u):=2-2\inf_{s\geq 0, \tau \in [0,u]}A_k(s,s+\tau)$.
\item[(b4)] The function $A_{\infty}(0,\tau)$ is non-increasing and satisfies that for any finite $h>0$ and $\theta\in (0,1)$
\begin{equation*}\label{con4}
a^2_{h,\theta}=\underset{0<t\leq h}{\inf} \left\{ \frac{A_{\infty}(0,\theta t) -A_{\infty}(0,t)}{1-A_{\infty}(0,t)} \right\} >0.
\end{equation*}
\end{itemize}
  Then we have
\begin{equation*}
\underset{k,T \rightarrow \infty}{\lim} \frac{1}{T}\log \pp \left(Z_t^{(k)}>0,\, \forall t\in [0,T] \right)=-b(A_{\infty}).
\end{equation*}
\end{itemize} 
\end{lem}
While Lemma \ref{lem1} shows the convergence of persistence exponent of general Gaussian processes under strict conditions of autocorrelation functions, Lemma \ref{ldm} below provides a lower bound on the persistence probability of a differentiable Gaussian process $Z(t)$, assuming a simple condition that the variances of $Z(t)$ and $Z'(t)$ are comparable. 
\begin{lem}  \label{ldm} \cite[Lemma 4.1]{DM} There is a universal constant $\mu \in (0,1)$, such that the following statements hold.
\begin{itemize}
\item[(i)] If $(Z_t)_{t \in [a,b]}$ is a differentiable centered
Gaussian process satisfying 
\begin{equation*}
2(b-a)^2 \sup \limits_{t \in [a,b]} \E (Z_t'^2) \leq \sup \limits_{t \in [a,b]} \E (Z_t^2), 
\end{equation*} 
then 
\begin{equation*}
\pp \left( \inf_{t \in [a,b]} Z_t > 0 \right) \geq \mu.
\end{equation*}
\item[(ii)] If $(Z_t)_{t \in [0, T]}$ is a differentiable centered Gaussian process with nonnegative autocorrelation function satisfying for all $t \leq T$
\begin{equation*}
2\vartriangle^2  \E (Z_t'^2) \leq \E (Z_t^2), 
\end{equation*} 
for some positive constant $\vartriangle$, then 
\begin{equation*}
\pp \left( \inf_{t \in [0,T]} Z_t > 0 \right) \geq \mu^{\lceil \frac{T}{\vartriangle}\rceil}.
\end{equation*}
\end{itemize}
\end{lem}  
\begin{proof}
Part (i) is exactly Lemma 4.1 in \cite{DM}. Part (ii) is a direct consequence of (i). Indeed, we divide the interval $[0, T]$ into $\lceil \tfrac{T}{\vartriangle}\rceil$ small intervals of length $\vartriangle$. Then the condition of (i) is verified in each small interval. Thus using Slepian's lemma and (i), we get (ii).
\end{proof}
\section{Proof of Theorem \ref{mth}}
In this section, we prove the main theorem, Theorem \ref{mth}, using preliminary lemmas in Section \ref{sec: prel}. The proof consists of three steps. In Subsection~\ref{sec: negligible} we show that the contribution to the persistence probability of two intervals $(0,n^{-1/6})$ and $(n^{1/6},\infty)$ is negligible. Then in Subsection \ref{sec3}, we compute the persistence exponent from the main interval $(n^{-1/6},n^{1/6})$. Finally, by bringing two previous steps together, we conclude the proof in Subsection \ref{sec: conc}.
\label{sec: proof}
\subsection{Negligible intervals}
\label{sec: negligible}
 In this part, we show that the contribution of intervals $(0,n^{-1/6})$ and $(n^{1/6},\infty)$ to the persistence exponent is negligible.
\begin{prop}  \label{pne}
We have
\begin{itemize}
\item[(i)] $\lim \limits_{n\rightarrow \infty} \frac{1}{\sqrt{n} }\log \pp \left(f_n(x)>0 \quad \forall x \in (0, n^{-1/6})\right) =0$,
\item[(ii)] $\lim \limits_{n\rightarrow \infty} \frac{1}{\sqrt{n} }\log \pp \left(f_n(x)>0 \quad \forall x \in (n^{1/6}, \infty)\right) =0$.
\end{itemize}
\end{prop}
\begin{proof}
Part (ii) is a consequence of (i). Indeed, we have
\begin{eqnarray*}
\pp \left(f_n(x)>0 \quad \forall x \in (n^{1/6}, \infty)\right) = \pp \left(f_n(x)>0 \quad \forall x \in (0, n^{-1/6})\right),
\end{eqnarray*}
since for $x>0$,
$$f_n(\tfrac{1}{x})=\frac{1}{x^n}\sum_{i=0}^{n} \binom{n}{i}a_ix^{n-i}\mathop{=}^{(\kL)}\frac{f_n(x)}{x^n}. $$ 
Now it remains to prove (i). Since the upper bound that $$\log \pp \left(f_n(x)>0 \quad \forall x \in (0,n^{-1/6})\right) \leq \log 1 =0$$ is trivial, we only need to show the lower bound
 \begin{eqnarray} \label{lp1}
 \liminf \limits_{n \rightarrow \infty} \frac{1}{\sqrt{n} }\log \pp \left(f_n(x)>0 \quad \forall x \in (0,n^{-1/6})\right) \geq 0.
 \end{eqnarray}
 By Slepian's lemma, \eqref{lp1} follows from the following lower bounds,
 \begin{eqnarray}
 && \liminf \limits_{n \rightarrow \infty} \frac{1}{\sqrt{n} }\log \pp \left(f_n(x)>0 \quad \forall x \in (0,\tfrac{\log n}{6n})\right) \geq 0, \label{lp1a}\\
&&  \liminf \limits_{n \rightarrow \infty} \frac{1}{\sqrt{n} }\log \pp \left(f_n(x)>0 \quad \forall x \in (\tfrac{\log n}{6n}, n^{-1/6})\right) \geq 0. \label{lp1b}
 \end{eqnarray}
 We first prove \eqref{lp1a}. We observe that 
 \begin{eqnarray}
 \pp \left(f_n(x)>0 \quad \forall x \in (0,\tfrac{\log n}{6n})\right) &\geq& \pp \left(a_0 > \Big|\sum_{i=1}^{n} \binom{n}{i}a_ix^{i} \Big| \quad \forall x \in (0,\tfrac{\log n}{6n})\right) \notag \\
 &\geq & \pp \left(a_0 > \max \limits_{1 \leq i \leq n} |a_i| \times \sum_{i=1}^{n} \binom{n}{i}x^{i}  \quad \forall x \in (0,\tfrac{\log n}{6n})\right) \notag \\
  &\geq & \pp \left(a_0 > \max \limits_{1 \leq i \leq n} |a_i| \times \left(1+ \tfrac{\log n}{6n} \right)^n \right) \notag \\
    &\geq & \pp \left(a_0 > \log n \times \left(1+ \tfrac{\log n}{6n} \right)^n \right) \times \pp \left(\max \limits_{1 \leq i \leq n} |a_i| \leq \log n \right) \notag \\
    &=& (1-\Phi(\xi_n)) \times (\Phi(\log n))^n, \label{h1}
 \end{eqnarray}
 where $\Phi(x)$ is the normal distribution function, and 
 $$\xi_n=\log n \times \left(1+ \tfrac{\log n}{6n} \right)^n.$$
 We notice that $\log (1-\Phi(x))=(\tfrac{-1}{2}+o(1))x^2$ as $x\rightarrow \infty$. Therefore, for $n$ large enough,
 \begin{eqnarray} \label{h2}
 (\Phi(\log n))^n \geq \left(1-e^{-\log^2 n/4 } \right)^n \geq 1/2,
 \end{eqnarray}
 and 
\begin{eqnarray} \label{h3}
\liminf \limits_{n\rightarrow \infty}\frac{\log (1-\Phi(\xi_n))}{\sqrt{n}} \geq \liminf \limits_{n\rightarrow \infty} \frac{-\xi_n^2}{\sqrt{n}} \geq \liminf \limits_{n\rightarrow \infty} \frac{- \log^2 n \times e^{\log n/3}}{\sqrt{n}} =0.  
 \end{eqnarray} 
 Combining \eqref{h1}, \eqref{h2} and \eqref{h3}, we get \eqref{lp1a}. We now prove \eqref{lp1b}. Let us define 
 $$g_n(x)=(x+1)^{-n}f_n(x).$$
 Then 
 \begin{eqnarray*}
 g_n'(x)&=&(x+1)^{-n} \left(f_n'(x)-\frac{n}{x+1}f_n(x) \right)\\
 &=& (x+1)^{-n} \sum \limits_{i=0}^n \binom{n}{i} a_i x^i \left( \tfrac{i}{x}-\tfrac{n}{x+1}\right).  
 \end{eqnarray*}
 Using Lemma \ref{lom}, we have 
 \begin{eqnarray} \label{eg}
 \E(g_n(x)^2)=(x+1)^{-2n} M_n(x) \geq  (x+1)^{-2n}   \binom{n}{i_x}^2x^{2i_x},
 \end{eqnarray}
 with 
$$i_x=\left[ \frac{nx}{x+1}\right].$$
Using the same arguments as in Lemma \ref{lom}, we can also prove that 
 \begin{eqnarray}
 \E(g_n'(x)^2) &=& (x+1)^{-2n} \sum \limits_{i=0}^n \binom{n}{i}^2  x^{2i} \left( \tfrac{i}{x}-\tfrac{n}{x+1}\right)^2 \notag\\
 &\leq&  (x+1)^{-2n}   3i_x^{3/4} \binom{n}{i_x}^2x^{2i_x}\left( \tfrac{i_x}{x}-\tfrac{n}{x+1}\right)^2. \label{egp}
 \end{eqnarray}
 Combining \eqref{eg} and \eqref{egp}, we obtain 
 \begin{eqnarray} \label{ggp}
 \E(g_n'(x)^2) \leq 3i_x^{3/4} \left( \frac{i_x}{x}-\frac{n}{x+1}\right)^2 \E(g_n(x)^2) \leq \frac{3n^{3/4}}{x^{5/4}} \E(g_n(x)^2).
 \end{eqnarray}
 Thus for $x\in (\tfrac{\log n}{6n},\tfrac{1}{\sqrt{n} })$,
 \begin{eqnarray*}
2 \vartriangle_{1,n}^2 \E(g_n'(x)^2) \leq  \E(g_n(x)^2),
 \end{eqnarray*}
 with 
 $$\vartriangle_{1,n}=\frac{(\log n) ^{5/8}}{8n}.$$
 Applying Lemma \ref{ldm}, we have
\begin{eqnarray*}
\pp \left(f_n(x)>0 \quad \forall x \in (\tfrac{\log n}{6n},\tfrac{1}{\sqrt{n}})\right) = \pp \left(g_n(x)>0 \quad \forall x \in (\tfrac{\log n}{6n},\tfrac{1}{\sqrt{n}})\right) \geq \mu^{\frac{1}{\vartriangle_{1,n} \sqrt{n}}}.
\end{eqnarray*}
Therefore, 
\begin{eqnarray} \label{tr1}
\liminf \limits_{n\rightarrow \infty}\frac{1}{\sqrt{n}} \log \pp \left(f_n(x)>0 \quad \forall x \in (\tfrac{\log n}{6n},\tfrac{1}{\sqrt{n}})\right) \geq \liminf \limits_{n\rightarrow \infty} \frac{\log \mu}{n \vartriangle_{1,n}} =0.
\end{eqnarray}
Using \eqref{ggp} for $x \in (n^{-1/2},n^{-1/6})$, we get 
\begin{eqnarray*}
2 \vartriangle_{2,n}^2 \E(g_n'(x)^2) \leq  \E(g_n(x)^2),
 \end{eqnarray*}
 with 
 $$\vartriangle_{2,n}=\frac{1}{3n^{11/16}}.$$
 Applying Lemma \ref{ldm}, we have
\begin{eqnarray*}
\pp \left(f_n(x)>0 \quad \forall x \in (n^{-1/2},n^{-1/6})\right) = \pp \left(g_n(x)>0 \quad \forall x \in (n^{-1/2},n^{-1/6}) \right) \geq \mu^{\frac{1}{n^{1/6}\vartriangle_{2,n} }}.
\end{eqnarray*}
Therefore, 
\begin{eqnarray} \label{tr2}
\liminf \limits_{n\rightarrow \infty}\frac{1}{\sqrt{n}} \log \pp \left(f_n(x)>0 \quad \forall x \in (n^{-1/2},n^{-1/6}) \right) \geq \liminf \limits_{n\rightarrow \infty} \frac{\log \mu}{n^{2/3} \vartriangle_{2,n}} =0.
\end{eqnarray}
Using \eqref{tr1}, \eqref{tr2} and Slepian's lemma, we get \eqref{lp1b}.
\end{proof}
 \subsection{The main interval}\label{sec3}
 We make a transformation 
 $$x=\tan^2\left(\frac{t}{2\sqrt{n}} \right).$$
 Then $x\in(n^{-1/6},n^{1/6})$ is equivalent to $t\in (\alpha_n, \pi\sqrt{n}-\alpha_n)$, with  
 $$\alpha_n=2 \sqrt{n} \tan^{-1}(n^{-1/12})=(2+o(1))n^{5/12}.$$
  Let us define for $t \in (\alpha_n, \pi\sqrt{n}-\alpha_n)$,
  $$h_n(t)=f_n(\tan^2(t/2\sqrt{n})).$$
 Then 
 \begin{eqnarray}
 \pp \left( f_n(x) >0 \quad \forall \, x \in (n^{-1/6},n^{1/6}) \right) =  \pp \left( h_n(t) >0 \quad \forall \, t \in (\alpha_n, \pi\sqrt{n}-\alpha_n) \right).
 \end{eqnarray}
 Moreover, the autocorrelation function of $h_n(t)$ is 
\begin{eqnarray}
B_n(t,s)= A_n \left( \tan^2(\tfrac{t}{2 \sqrt{n}}), \tan^2( \tfrac{s}{2 \sqrt{n}} )\right) = \frac{M_n \left(\tan(\tfrac{t}{2 \sqrt{n}}) \tan( \tfrac{s}{2 \sqrt{n}} ) \right)}{\sqrt{M_n(\tan^2(\tfrac{t}{2 \sqrt{n}})}) \sqrt{M_n(\tan^2(\tfrac{s}{2 \sqrt{n}})} )}.
\end{eqnarray} 
We recall the approximation on $M_n(u)$. For $u\in(n^{-1/6},1)$, 
\begin{equation} \label{mnu}
M_n(u)=(1+O(n^{-1/24}))\frac{(u+1)^{2n+1}}{\sqrt{\pi n u}}.
\end{equation} 
  For $u \in (1, \infty)$, we remark that 
  $$M_n(u) = u^{2n} M_n(\tfrac{1}{u}).$$
  Therefore, the estimate  \eqref{mnu} holds for all $u\in (n^{-1/6},n^{1/6})$. Hence,
\begin{eqnarray*}
B_n(t,s) &=&  \frac{M_n \left(\tan(\tfrac{t}{2 \sqrt{n}}) \tan( \tfrac{s}{2 \sqrt{n}} ) \right)}{\sqrt{M_n(\tan^2(\tfrac{t}{2 \sqrt{n}})}) \sqrt{M_n(\tan^2(\tfrac{s}{2 \sqrt{n}})} )} \\
&=& (1+O(n^{-1/24}))\left( \frac{(1+ \tan(\tfrac{t}{2 \sqrt{n}}) \tan( \tfrac{s}{2 \sqrt{n}} ))^2}{ \left(1+\tan^2(\tfrac{t}{2 \sqrt{n}}) \right)\left(1+\tan^2(\tfrac{s}{2 \sqrt{n}}) \right)}\right)^{\tfrac{2n+1}{2}}\\
&=&(1+O(n^{-1/24})) \left[ \cos \left(\tfrac{t-s}{2\sqrt{n}} \right)\right]^{2n+1}.
\end{eqnarray*}   
We shift the interval $(\alpha_n, \pi \sqrt{n} -\alpha_n)$ to the interval $(0,\pi \sqrt{n}-2 \alpha_n)$ by changing variable
$$u=t-\alpha_n,$$
and define 
$$\bar{h}_n(u)=h_n(u+\alpha_n).$$
Then the autocorrelation of $\bar{h}_n(u)$ is 
\begin{eqnarray} \label{bbuv}
\bar{B}_n(u,v) = B_n(u+ \alpha_n, v+\alpha_n) = (1+O(n^{-1/24})) \left[ \cos \left(\tfrac{u-v}{2\sqrt{n}} \right)\right]^{2n+1}.
\end{eqnarray}
Observe that for fixed $u,v$, 
\begin{eqnarray*}
\bar{B}_n(u,v) \rightarrow e^{-(u-v)^2/4}. 
\end{eqnarray*} 
This fact suggests us to verify Conditions (b1)-(b4) of Lemma \ref{lem1} for the sequence of Gaussian processes $\{Z_t^{(n)}\}$ with autocorrelation functions $\bar{B}_n(u,v)$ and the limit stationary Gaussian process $Z_t^{(\infty)}$ with autocorrelation function $e^{-(u-v)^2/4}$.

\noindent {\it Verification of the condition (b1).} It is easily deduced from (\ref{bbuv}).

\noindent {\it Verification of the condition (b2).}  Here we consider two cases. In the first case where $\tau/\sqrt{n} \rightarrow 0$, as $\tau, n \rightarrow \infty$,
\begin{eqnarray*}
\tau^{\alpha} \bar{B}_n(u,u+\tau)  \leq (const)  \tau^{\alpha} \left[ \cos \left(\tfrac{\tau}{2\sqrt{n}} \right)\right]^{2n+1} \leq (const) \tau^{\alpha} e^{-\tau^2/4} \rightarrow 0.
\end{eqnarray*}

In the second case, assume that there exists a positive constant $c_0$ such that $\tau/\sqrt{n} \geq c_0$, then as $\tau, n \rightarrow \infty$,
\begin{eqnarray*}
\tau^{\alpha} \bar{B}_n(u,u+\tau)  \leq (const)  n^{\alpha/2} \left[ \cos \left(c_0/2 \right)\right]^{2n+1} \rightarrow 0.
\end{eqnarray*}

Thus Condition (b2) is verified.

\noindent {\it Verification of Condition (b3)}.  Using \eqref{bbuv}, we have for $n$ large enough
\begin{eqnarray*}
\bar{B}_n(t,t+\tau) \geq (1-n^{-1/30}) \left[\cos \left( \tfrac{\tau}{2 \sqrt{n}}\right)\right] ^{2n+1} \geq (1-n^{-/30})(1-\tau^2).
\end{eqnarray*}
Therefore, 
\begin{eqnarray*}
\bar{p}^2_n(w) = 2 -2 \inf_{\substack{0\leq \tau\leq w \\ 0\leq t, t + \tau   \leq \pi \sqrt{n} -2 \alpha_n}} \bar{B}_n(t, t+ \tau) \leq 2 w^2 + 2 n^{-1/30}.
\end{eqnarray*}
Hence, for any $\delta >0$, as $w\rightarrow 0$
\begin{eqnarray*}
|\log w|^2\sup \limits_{n \geq w^{-\delta}} \bar{p}_n^2(w) \rightarrow 0.
\end{eqnarray*}
Thus, to verify (b3), it suffices to show that for some $\delta >0$
\begin{eqnarray}  \label{nbe}
\lim \limits_{w \rightarrow 0^+}|\log w|^2\sup \limits_{n \leq w^{-\delta}} \bar{p}_n^2(w) < \infty.
\end{eqnarray}
To show \eqref{nbe} holds, it is sufficient to prove that 
\begin{eqnarray}  \label{nbeu}
\lim \limits_{u \rightarrow 0^+}|\log u|^2\sup \limits_{n \leq u^{-\delta}} p_n^2(u) < \infty,
\end{eqnarray}
where 
\begin{eqnarray*}
p^2_n(u) = 2 -2 \inf_{{\substack{0\leq y-x \leq u \\n^{-1/6} \leq x \leq y \leq n^{1/6}  }}} A_n(x, y).
\end{eqnarray*}
Recall that 
\begin{eqnarray*}
A_n(x,y)= \frac{M_n(\sqrt{xy})}{\sqrt{M_n(x)}\sqrt{M_n(y)}},
\end{eqnarray*}
where 
$$M_n(x)=\sum \limits_{i=0}^n \binom{n}{i}^2x^{2i}.$$
We have 
\begin{eqnarray*}
\varepsilon_n(x,y):= M_n(\sqrt{xy}) - M_n(x)= \sum \limits_{i=0}^n \binom{n}{i}^2x^{i}(y^i-x^i) \geq 0,
\end{eqnarray*}
since $y\geq x$, and 
\begin{eqnarray*}
\tilde{\varepsilon}_n(x,y):= M_n(y) - M_n(x)- 2 \varepsilon_n(x,y) =\sum \limits_{i=0}^n \binom{n}{i}^2(y^i-x^i)^2 \geq 0.
\end{eqnarray*}
Let $a >0 $ and $b,c \geq 0$ be  real numbers satisfying $ac-b^2 \geq 0$. Then
\begin{eqnarray*}
1-\frac{a+b}{\sqrt{a(a+2b+c)}} = \frac{ac-b^2}{\left(a+b+\sqrt{a(a+2b+c)}\right)\sqrt{a(a+2b+c)}}\leq \frac{ac-b^2}{a^2}.
\end{eqnarray*}
By the Cauchy-Schwarz inequality, $M_n(x)\tilde{\varepsilon}_n(x,y)-\varepsilon^2_n(x,y)\geq 0$. Hence, using the above inequality for $a=M_n(x)$,  $b=\varepsilon_n(x,y)$ and $c=\tilde{\varepsilon}_n(x,y)$, we get
\begin{eqnarray} \label{rv4}
0 \leq 1- A_n(x,y) &=& 1- \frac{M_n(x)+\varepsilon_n(x,y)}{\sqrt{M_n(x)\left( M_n(x)+ 2 \varepsilon_n(x,y) +   \tilde{\varepsilon}_n(x,y)\right)}} \notag\\
&\leq& \frac{M_n(x) \tilde{\varepsilon}_n(x,y) -\varepsilon_n^2(x,y)}{M_n^2(x)} \notag \\
&=& (y-x)^2 \frac{M_n(x) \tilde{\varepsilon}_{1,n}(x,y) -\varepsilon_{1,n}^2(x,y)}{M_n(x)},  \label{epp}
\end{eqnarray}
with 
\begin{eqnarray*}
\varepsilon_{1,n}(x,y) =\sum \limits_{i=0}^n \binom{n}{i}^2x^{i} \left(\frac{y^i-x^i}{y-x}\right),
\end{eqnarray*}
and
\begin{eqnarray*}
\tilde{\varepsilon}_{1,n}(x,y) =\sum \limits_{i=0}^n \binom{n}{i}^2 \left(\frac{y^i-x^i}{y-x}\right)^2.
\end{eqnarray*}
Using the fact  that 
\begin{eqnarray*}
\sum \limits_{i=0}^n a_i^2 \sum \limits_{i=0}^n b_i^2 - \left(\sum \limits_{i=0}^n a_i b_i \right)^2= \sum \limits_{i<j}^n (a_ib_j-a_jb_i)^2,
\end{eqnarray*}
we get  
\begin{eqnarray} \label{cs}
M_n(x) \tilde{\varepsilon}_{1,n}(x,y) -\varepsilon_{1,n}^2(x,y) = \sum_{i<j} \binom{n}{i}^2 \binom{n}{j}^2x^{2i}y^{2i}\left(\frac{y^{j-i}-x^{j-i}}{y-x}\right)^2.
\end{eqnarray}
We notice  that 
\begin{eqnarray} \label{rv5}
0\leq y-x \leq \tau \leq u, \hspace{1 cm} n\leq u^{-\delta}, \hspace{1cm} x \geq n^{-1/6} \geq u^{\delta/6}.
\end{eqnarray}
Hence,
\begin{eqnarray} \label{rv1}
y^{2i}\leq(x+u)^{2i}=x^{2i} \left(1+\frac{u}{x}\right)^{2i} \leq x^{2i}\left(1+n^{\tfrac{1}{6} -\tfrac{1}{\delta}}\right)^{2n} \leq 2 x^{2i},
\end{eqnarray}
for all $\delta \leq 1/2$ and $n$ large enough. Moreover, since $y-x \leq u \leq x$,
\begin{eqnarray} \label{rv2}
\frac{y^{j-i}-x^{j-i}}{y-x}=\sum_{k=0}^{j-i-1}(y-x)^kx^{j-i-k} \leq (j-i)x^{j-i-1} \leq nx^{j-i-1}.
\end{eqnarray}
Combining \eqref{cs}, \eqref{rv1} and \eqref{rv2} yields that
\begin{eqnarray} \label{rv3}
M_n(x) \tilde{\varepsilon}_{1,n}(x,y) -\varepsilon_{1,n}^2(x,y) &\leq& \frac{2n^2}{x^2}\sum_{i<j} \binom{n}{i}^2 \binom{n}{j}^2x^{2i+2j} \notag\\
&\leq& \frac{n^2}{x^2}M_n^2(x).
\end{eqnarray}
It follows from \eqref{rv4}, \eqref{rv5} and \eqref{rv3} that
\begin{eqnarray*}
0\leq 1-A_n(x,y) &\leq& \frac{u^2n^2}{x^2} \leq u^{2-7\delta/3} \leq u^{5/6}, 
\end{eqnarray*}
for $\delta \leq 1/2$. As consequence, \eqref{nbeu} holds. 

\noindent {\it Verification of the condition (b4).} It is easy to check (or see Remark 3.1 in \cite{LS}).

By the validity of Conditions (b1)-(b4), we deduce  from Lemma \ref{lem1} the following proposition. 
\begin{prop} \label{pma} We have 
$$ \lim \limits_{n\rightarrow \infty} \frac{1}{\pi \sqrt{n} }\log \pp \left(f_n(x)>0 \quad \forall x \in ( n^{-1/6},n^{1/6})\right) = -b, $$
with $b$ is as in the statement of Theorem \ref{mth}.
\end{prop}
  \subsection{Conclusion}
\label{sec: conc}  
   Thanks to Slepian's inequality, using  Propositions \ref{pne} and \ref{pma}  we get
$$ \liminf \limits_{n\rightarrow \infty} \frac{1}{ \pi \sqrt{n} }\log \pp \left(f_n(x)>0 \quad \forall x \in (-\infty, \infty)\right) \geq -b. $$  
 On the other hand, it follows directly from Proposition \ref{pma} that 
  $$ \limsup \limits_{n\rightarrow \infty} \frac{1}{ \pi \sqrt{n} }\log \pp \left(f_n(x)>0 \quad \forall x \in (-\infty, \infty)\right) \leq -b. $$  
  Combining these two inequalities we get Theorem \ref{mth}.
 \section{Summary and future work }
 In this paper, we have obtained an asymptotic formula for the persistence probability of the random polynomial $f_n$ in \eqref{eq: fn} that arises from evolutionary game theory. The persistence probability corresponds to the probability that a symmetric $n$-player two-strategy random game has no internal equilibria. We note that $f_n$ forms a different class of random polynomials that have been studied extensively in the literature particularly in random polynomial theory, see  \cite{EK95} and a recent paper \cite{lubinsky2018} and references therein for information. There are several open problems that are of interest for both evolutionary game theory and random polynomial theory that we do not address in this paper  such as proving a central limit theorem and a large deviation principle for the empirical measures of the real zeros of $f_n$ as well as studying universality phenomena for this class of random polynomials. We leave these problems for future research.
 \label{sec: final}
 \section{Appendix: derivation of $f_n$ from evolutionary game theory}
\label{sec: replicator}
In this appendix, we review the derivation of the random polynomial $f_n$ in \eqref{eq: fn} from the replicator dynamics for multi-player two-strategy games in evolutionary game theory. The replicator equation for multi-player two-strategy games has already been derived in previous works~ \cite{hofbauer:1998mm,sigmund:2010bo,gokhale:2010pn}. For the sake of completeness,  we rederive it here. 

We consider an infinitely large population consists of individuals using two strategies, $A$ and $B$. Let $y$, $0 \leq y\leq 1$, be the frequency of strategy $A$ in the population. The frequency of strategy $B$ is thus $(1-y)$. The interaction of the individuals in the population is in randomly selected groups of $n$ participants, that is, they interact and obtain their fitness from  $n$-player games. In this paper, we consider symmetric games where the payoffs do not depend on the ordering of the players. Let $a_k$ (respectively, $b_k$) be the payoff of that an $A$-strategist (respectively, $B$) achieves when interacting with a group containing  $k$ $A$ strategists (and $n-k$ $B$ strategists).  In symmetric games, the probability that an $A$ strategist interacts with $k$ other $A$ strategists in a group of size $n$ is 
\begin{equation*}
\begin{pmatrix}
n-1\\
k
\end{pmatrix}y^k (1-y)^{n-1-k}.
\end{equation*}
We note that this probability depends only on the number $k$ of $A$ strategist but not on the particular order of the group. The average payoffs of $A$ and $B$  are, respectively 
\begin{equation*}
\pi_A= \sum\limits_{k=0}^{n-1}a_k\begin{pmatrix}
n-1\\
k
\end{pmatrix}y^k (1-y)^{n-1-k},	\quad
\pi_B = \sum\limits_{k=0}^{n-1}b_k\begin{pmatrix}
n-1\\
k
\end{pmatrix}y^k (1-y)^{n-1-k}.
\end{equation*}
The replicator equation of a $d$-player two-strategy game is given by  \cite{hofbauer:1998mm,sigmund:2010bo,gokhale:2010pn}
\begin{equation*}
\dot{y}=y(\pi_A-\overline{\pi})=y(1-y)\big(\pi_A-\pi_B\big),
\end{equation*}
where $\overline{\pi}:=y\pi_A+(1-y)\pi_B$ is the average payoff of the population. The replicator equation reflects the natural selection. In fact, if $\pi_A\geq \overline{\pi}$ then $y$ increases, that is $A$ spreads in the population; vice versa, if $\pi_A<\overline{\pi}$ then $y$ decreases and $A$ declines.
Equilibrium points of the dynamics satisfy that $y(1-y)(\pi_A-\pi_B)=0$. Since $y=0$ and $y=1$ are two trivial equilibrium points, we focus only on internal ones,  i.e. $0 < y < 1$. They satisfy the condition that the  fitnesses of both strategies are the same $\pi_A=\pi_B$, which gives rise to
\begin{equation*}
\sum\limits_{k=0}^{n-1}\beta_k \begin{pmatrix}
d-1\\
k
\end{pmatrix}y^k (1-y)^{n-1-k} = 0,
\end{equation*}
where $\beta_k = a_k - b_k$. Using the transformation $x= \frac{y}{1-y}$, with $0< x < +\infty$, dividing the left hand side of the above equation by $(1-y)^{n-1}$ we obtain the following polynomial equation for $x$
\begin{equation}
\label{eq: eqn for y}
\sum\limits_{k=0}^{n-1}\beta_k\begin{pmatrix}
n-1\\
k
\end{pmatrix}x^k=0.
\end{equation}
Note that this equation can also be derived from the definition of an evolutionary stable strategy, see e.g., \cite{broom:1997aa}. In complex large systems, information about the interaction between participants is rarely available at the level of detail sufficient for the exact computation of the payoff matrix; therefore, it is necessary to suppose that the payoff matrix entries $a_k$ and $b_k$ (thus $\beta_k$) for $0\leq k\leq n-1 $, are random variables. We then obtain random games and the expression on the right-hand side of \eqref{eq: eqn for y} becomes a random polynomial. It is exactly the random polynomial $f_{n-1}$ in \eqref{eq: fn} that we start with.

We note that in this paper we need to make an assumption that the payoff differences $a_k-b_k$ are independent
standard normal random variables. This choice could be interpreted as modeling noise added to payoffs of a game where both strategies are neutral, i.e., have always identical payoffs. Although this assumption is rather restricted, it provides an exact match between random polynomial theory and random polynomial theory which opens up a new avenue for future research, for instance, among other things, to relax the identical and independent assumption.
 
 \begin{ack} \emph{The first author is supported by the fellowship of the Japan Society for the Promotion of Science and the Grant-in-Aid for JSPS fellows Number 	17F17319. We would like to thank anonymous referees for their useful suggestions.}
 \end{ack}
 \bibliographystyle{alpha}
 \bibliography{GTbib}
\end{document}